\documentclass[11pt]{amsart}

\usepackage{amssymb, amsfonts}
\usepackage{mathtools}
\usepackage{graphicx}
\usepackage[hidelinks]{hyperref}
\usepackage{amsaddr}

\usepackage[margin=1.25in]{geometry}
\usepackage{setspace}
\usepackage{parskip}
\theoremstyle{plain}
\newtheorem{thm}{Theorem}

\newtheorem{lem}{Lemma}[section]
\newtheorem{prop}{Proposition}[section]

\numberwithin{equation}{section}

\newcommand{\E}[1]{\mathbf{E}\left[#1\right]}
\newcommand{\pr}[1]{\mathbf{Pr}\left[#1\right]}
\newcommand{\ind}[1]{\mathbf{1}_{\{ #1 \}}}
\newcommand{\eps}{\varepsilon}
\newcommand{\Z}{\mathbb{Z}}
\newcommand{\R}{\mathbb{R}}
\newcommand{\Ai}{\mathrm{Ai}}
\newcommand{\A}{\mathcal{A}}
\newcommand{\h}{\mathfrak{h}}

\title{TASEP fluctuations with soft-shock initial data}

\author{Jeremy Quastel}

\address{Department of Mathematics, University of Toronto, \\ Toronto, ON, Canada.}
\email{quastel@math.toronto.edu}

\author{Mustazee Rahman}

\address{Department of Mathematics, Massachusetts Institute of Technology, \\ Cambridge, MA, USA \\
and \\
Department of Mathematics, Durham University, \\ Durham, UK.}
\email{mustazee@gmail.com}

\keywords{Burgers equation, exclusion process, KPZ universality, shock fluctuations, Tracy-Widom law}
\subjclass[2010]{Primary: 60K35; Secondary: 46N30, 82C22, 82C23.}

\begin{document}

\begin{abstract}
\normalsize

We consider the totally asymmetric simple exclusion process with
\emph{soft-shock} initial particle density, which is a step function increasing in the direction
of flow and the step size chosen small to admit KPZ scaling. The initial
configuration is deterministic and the dynamics create a shock.

We prove that the fluctuations of a particle at the macroscopic position of the shock converge
to the maximum of two independent GOE Tracy-Widom random variables, which establishes a conjecture
of Ferrari and Nejjar. Furthermore, we show the joint fluctuations of particles near the shock are determined
by the maximum of two lines described in terms of these two random variables.
The microscopic position of the shock is then seen to be their difference.

Our proofs rely on determinantal formulae and  a novel factorization of the associated kernels.
\end{abstract}

\maketitle
\newpage

\section{Introduction}

The continuous time totally asymmetric simple exclusion process (TASEP) is an
interacting particle system on an one-dimensional lattice. At time zero, there is a given
initial configuration of particles such that every site of $\Z$ occupies at most one particle.
The dynamics are as follows. A particle jumps randomly to its neighbouring site to the right
provided that it is empty. The jumps of a particle are independent of the others' and performed
with exponential waiting times having mean 1.

The observables in TASEP are the positions of particles. Initially, particles are numbered
by integers in increasing order from right to left with particle 1 being the first one to
the left of the origin. Denote by $X_t(n)$ the position of particle number $n$ at time $t$.
Thus, given the initial configuration $X_0(\cdot)$, particles are numbered such that
$$ \cdots < X_0(3) < X_0(2) < X_0(1) < 0 \leq X_0(0) < X_0(-1) < X_0(-2) < \cdots .$$
By convention, if there is a rightmost particle then $X_0(n) = + \infty$ for every $n$ after
that particle, and similarly, $X_0(n) = - \infty$ for every $n$ preceding a leftmost particle.
As an example, if the initial configuration occupies all sites at the negative integers then
$X_0(1) = -1, X_0(2) = -2, X_0(3) = -3$ and so on, while $X_0(0) = X_0(-1) = \cdots = + \infty$.

A detailed construction and basic features of TASEP are given in \cite{Lig}.
The model can also be seen as a randomly growing one-dimensional interface whose gradient
is the particle density. In this respect it belongs to the Kardar-Parisi-Zhang (KPZ) universality class.
The surveys \cite{BG, Q} discuss the relationship of TASEP to KPZ.

Despite its simplicity TASEP displays many of the interesting behaviour of non-equilibrium statistical mechanics.
Consider a deterministic initial configuration such that the macroscopic particle density is
$\rho_{-}$ to the left of the origin and $\rho_{+}$ to the right:
\begin{equation} \label{eqn:ID}
\rho_{\pm} = \lim_{t \to \, \infty}\; \frac{\# \, \{\text{particles in the interval} \; [0, \pm t]\;\text{at time}\;0\}}{t}\, .
\end{equation}
For instance, particles may be arranged periodically in large enough blocks to attain such a profile.
On the macroscopic scale the evolution of the particle density is a solution to Burgers' equation \cite{R, S3}.
Namely, for every $T > 0$ there is a density $u(T,x)$ such that
$$\int \limits_a^b u(T,x) \, dx = \lim_{t \to \, \infty}\; \frac{\# \, \{\text{particles in the interval} \; [at, bt]\;\text{at time}\; Tt\} }{t} \;\; \text{almost surely},$$
and $u$ is the unique entropy solution of Burgers' equation:
\begin{equation} \label{eqn:Burgers}
\partial_T u + \partial_x (u (1-u)) = 0, \quad u(0,x) = \rho_{-} \ind{x < 0} + \rho_{+} \ind{x \geq 0}.
\end{equation}

When $\rho_{-} < \rho_{+}$, there is a traffic jam in the system because particles to the left of the origin,
moving at macroscopic speed $1-\rho_{-}$, run into particles to the right of the origin moving at a slower
speed of $1-\rho_{+}$. In this case the relevant solution of \eqref{eqn:Burgers} is given by the travelling front
$$u(T,x) = u(0, x-\nu T), \;\; \text{where}\;\; \nu = 1- \rho_{-} - \rho_{+}\,.$$

The number $\nu$ is the speed of the traffic jam. This is the shock in Burgers' equation.
It is of interest to study the microscopic features of the shock, ergo, the fluctuations
of TASEP with an initial particle configuration as above.

\begin{figure}[t]
\begin{center}
\includegraphics[scale=0.45]{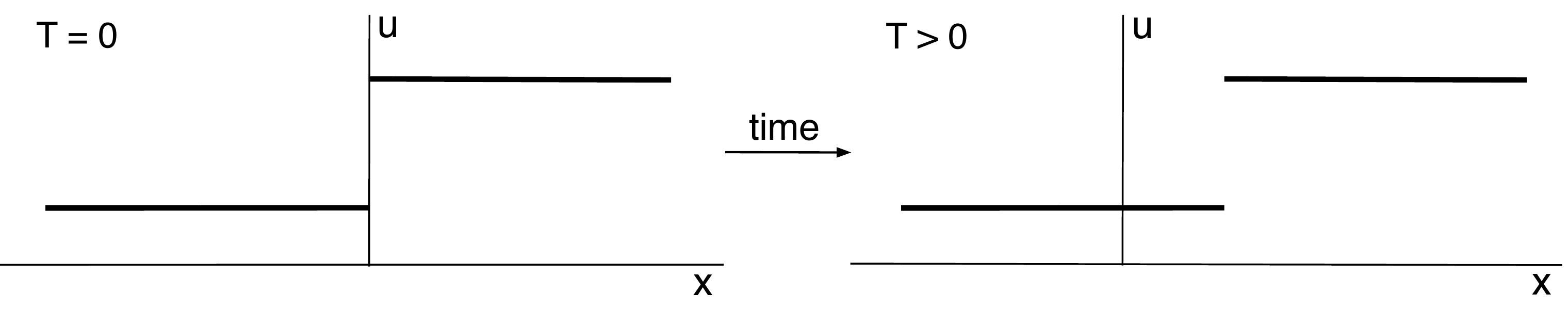}
\caption{The travelling shock-front $u(T,x)$ that solves \eqref{eqn:Burgers}.}
\end{center}
\end{figure}

A proxy for the location of the shock is the particle at macroscopic position $\nu t$.
For large times, the number\footnote{Rounding particle numbers to nearest integers
is omitted throughout the paper.} of said particle is
$$n^{\rm{shk}}_t = (\rho_{-} \rho_{+})t \,.$$
Its position fluctuates randomly to the order $t^{1/3}$ and one would like to calculate, for every $a \in \R$,
$$\lim_{t \to \, \infty}\, \pr{ X_t \big(n^{\rm{shk}}_t \big) \geq \nu t - a t^{1/3}}\,.$$

\subsection{The soft shock}
This paper considers a softening of the shock where the parameters $\rho_{\pm}$ are scaled as
\begin{equation} \label{eqn:softrho}
\rho_{\pm} = \frac{1 \pm \beta (t / 2)^{-1/3}}{2},\;\;  \beta \in \R \;\;\text{and}\;\; t \geq 2|\beta|^{3}.
\end{equation}
In the soft shock scenario, TASEP is run until time $t$ with the choice of $\rho_{\pm}$ as in \eqref{eqn:softrho}
and $t$ being the parameter within $\rho_{\pm}$. One then considers the law of $X(n^{\rm{shk}}_t)$
in the double limit as $t \to \infty$ followed by $\beta \to \infty$, in order to transition into the shock.

\begin{thm} \label{thm:1}
Consider TASEP with a deterministic initial configuration of particles having macroscopic density
as in \eqref{eqn:ID} and $\rho_{\pm}$ scaled as in \eqref{eqn:softrho}. Then,
$$\lim_{\beta \to \infty}\; \lim_{t \to \infty}\, \pr{X_t \big(n^{\rm{shk}}_t \big) \geq - a t^{1/3}} = F_1(2a)^2,$$
where $F_1$ is the distribution function of the GOE Tracy-Widom law.
\end{thm}

The soft shock is introduced in \cite{FN2} and Theorem \ref{thm:1} proves a conjecture there.
The scaling \eqref{eqn:softrho} is considered ``critical" for TASEP since the resulting particle
fluctuations belong to the KPZ universality class with regards to their scaling exponents and limit laws.
Interestingly, as remarked in Section \ref{sec:shockremarks}, the large time limit of the soft shock
recovers many of the universal Airy processes ($\rm{Airy}_1$, $\rm{Airy}_2$, etc.) through various
limit transitions of the parameter $\beta$. Also, the rate of convergence that we
find in Theorem \ref{thm:1} is of order $\beta^{-1}$, although this is not optimized.

The advantage of the soft shock is that it allows to transition into the hard shock
by means of exact calculation of statistical laws. More precisely, one can describe the
limiting law of $X_t(n^{\rm{shk}}_t)$ as $t \to \infty$ in terms of Fredholm determinants.
Indeed, Theorem \ref{thm:2} provides the large $t$ limiting joint distribution of particles that
are in the window of the soft shock, and Theorem \ref{thm:3} establishes the large $\beta$ limit of that.
Together, they imply Theorem \ref{thm:1}. These methods should also apply to prove GOE Tracy-Widom
cubed, quadrupled, etc., limiting laws at the merger of shocks when the initial particle density has two
jumps, three jumps, and so on. We do not pursue it here.

\subsection{Large time limit of the TASEP with soft shock}
In the case of soft shock the particle numbered $n^{\rm{shk}}_t$ has non-trivial correlations with
other particles that are within a distance of order $t^{2/3}$ of its position.  Their positions fluctuate
on a scale of order $t^{1/3}$. As such, consider particles having numbers
$$n(t, x) = n^{\rm{shk}}_t - x(t/2)^{2/3} = \frac{t}{4} - \frac{\beta^2}{2}(t/2)^{1/3} - x(t/2)^{2/3} ,$$
for $x \in \R$, which at time $t$ have macroscopic positions
$$m(t,x) = \frac{x (t/2)^{2/3}}{\rho_{-}} = \frac{2x}{1-\beta(t/2)^{-1/3}}\, (t/2)^{2/3}.$$
The first limit transition derives the large $t$ limit of the process
\begin{equation} \label{eqn:firstlim} x \, \mapsto \, \frac{X_t(n(t,x)) - m(t,x)}{-(t/2)^{1/3}}\,.\end{equation}

\begin{thm} \label{thm:2}
Given real numbers $x_1 < x_2 < \cdots < x_m$ and $a_1, \ldots, a_m$, as $ t \to \infty$,
$$\pr{ X_t(n(t,x_i)) \geq m(t,x_i) - a_i(t/2)^{1/3},\; 1 \leq i \leq m}$$
converges to
$$\pr{ \h(1,x_i \,; \, 2\beta|y|) \leq \beta^2 - 2\beta x_i + a_i, \; 1 \leq i \leq m},$$
where $\h(1,x; 2\beta|y|)$ is a random function of the variable $x$.
The multi-point distribution functions of $\h(1,x; 2\beta |y|)$ are given
in terms of Fredholm determinants:
\begin{align*}
& \pr{\h(1,x_i \,;\, 2\beta|y|) \leq a_i,\; 1\leq i \leq m} = \\
& \quad \det \left ( I - e^{-x_m \partial^2}K_{\beta}e^{x_m \partial^2}
\big (I - e^{(x_1-x_m)\partial^2}\bar{\chi}_{a_1} e^{(x_2-x_1)\partial^2} \bar{\chi}_{a_2}
\cdots e^{(x_m-x_{m-1})\partial^2}\bar{\chi}_{a_m} \big)\right)_{L^2(\R)}
\end{align*}
where $\bar{\chi}_{a}(u) = \ind{u \leq a}$ is projection onto $L^2(-\infty, a)$ and
$K_{\beta}$ is an explicit operator.
\end{thm}
$K_{\beta}$ is defined separately in Section \ref{sec:softkernel} since its introduction
requires crucial terminology and concepts.

Complicated though the determinant in Theorem \ref{thm:2} may appear, observe the
one-point distribution functions of $\h(1,x; 2\beta |y|)$ are given by the
Fredholm determinant of operators $e^{-x\partial^2} K_{\beta} e^{x \partial^2}$
over the spaces $L^2(a,\infty)$. These will turn out simpler and play a crutial role in the proofs.

The reason we call the limit process $\h(1,x; 2\beta |y|)$ is that it is the height function at time $1$
of the KPZ fixed point with initial data $h_0(y) = 2\beta|y|$, as introduced in \cite{MQR}. 
The KPZ fixed point refers to what is expected to be the asymptotic scaling invariant Markov process
for the KPZ universality class. Although the KPZ fixed point motivates our paper to an extent,
the kernels in this case were actually known previously in \cite{FN2}, and so the results used from
\cite{MQR} are somewhat auxillary.

\subsection{Transition into the shock}
The main result of the paper is the large $\beta$ limit law of the process $\h(1,x; 2\beta|y|)$.
\begin{thm} \label{thm:3}
As $\beta \to \infty$, the process
$$ x \mapsto \, \h \big (1, (2\beta)^{-1} x \, ; \, 2 \beta |y| \big ) - \beta^2$$
converges in the sense of finite dimensional laws to the process
\begin{equation}\label{limitprocess}
x \mapsto \, \max \{ \, 2^{-2/3} X_{\rm{TW}_1} - x,\, 2^{-2/3} X'_{\rm{TW}_1} + x \, \},
\end{equation}
where $X_{\rm{TW}_1}$ and $X'_{\rm{TW}_1}$ are two independent GOE Tracy-Widom random variables.
\end{thm}

Stated in terms of the TASEP soft shock, Theorem \ref{thm:3} asserts that in the double limit
of $t \to \infty$ followed by $\beta \to \infty$ the process
$$x \, \mapsto \, \frac{X_t \big (n(t,(2\beta)^{-1}x) \big) - \beta^{-1}x(t/2)^{2/3}}{-(t/2)^{1/3}}$$
converges in law to the process \eqref{limitprocess}. Process \eqref{limitprocess} may be
thought of as the asymptotic ``shock process'' of TASEP with initial density \eqref{eqn:ID}.

\begin{figure}[htpb]
\begin{center}
\includegraphics[scale=0.6]{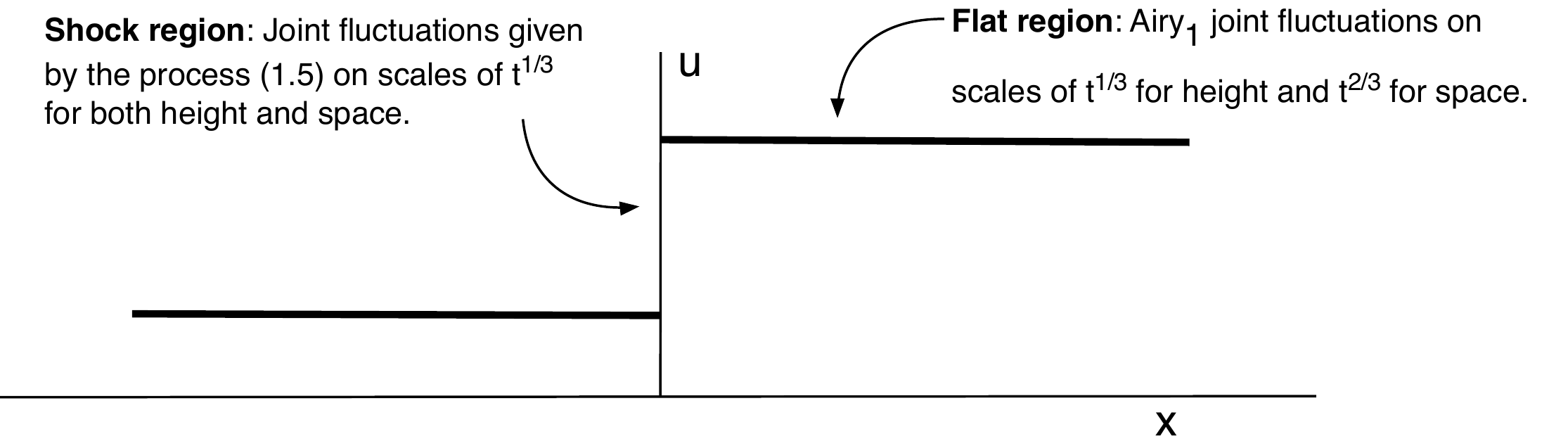}
\caption{Fluctuations of TASEP that arise from initial density \eqref{eqn:ID}
when $\rho_{\pm}$ are given by \eqref{eqn:softrho}.}
\label{fig:2}
\end{center}
\end{figure}

\paragraph{\textbf{A remark on interpretation of Theorem \ref{thm:3}}}
The process \eqref{limitprocess} may be expressed as $|x- X| + Y$ with
$X = (X_{\rm{TW}_1} - X'_{\rm{TW}_1})/2^{5/3}$ and $Y = (X_{\rm{TW}_1} + X'_{\rm{TW}_1})/2^{5/3}$.
As this process represents the asymptotic position of particles around the shock, the ``microscopic density"
near the shock may be interpreted as its derivative, which is simply an increasing step function with jump at
$X = (X_{\rm{TW}_1} - X'_{\rm{TW}_1})/2^{5/3}$. The \emph{microscopic position} of the shock is then
at $X$, and one finds the TASEP shock to be rigid in that the microscopic density remains a step with the randomness
only affecting its microscopic location.

\subsection{Remarks on the soft shock process} \label{sec:shockremarks}
``Soft shock" is bit of a misnomer since the shock manifests for large values of $\beta$ whereas
the process in Theorem \ref{thm:2} has interesting features for negative values of $\beta$ as well.
We have a family of processes interpolating from the $\rm{Airy}_2$ process at
$\beta = -\infty$ to the process \eqref{limitprocess} at $\beta = + \infty$.
This is easily seen from the framework of the aforementioned KPZ fixed point, as the
mapping from initial data $h_0$ to $\h(1,x; h_0)$ is continuous, so long as $h_0$ is upper semicontinuous
with values in $[-\infty, \infty)$ and bounded from above by a linear function.

When $\beta = 0$, $\h(1,x; 0)$ is the $\rm{Airy}_1$ process corresponding to flat initial data $h_0 \equiv 0$.
As $\beta \to - \infty$, the function $2\beta|y|$ converges to $- \infty \, \ind{y \neq 0}$, which is
called the narrow wedge or droplet ($-\infty \times 0 = 0$). Then $\h(1,x; 2\beta |y|)$ converges to
$\h(1,x; \, \text{narrow\;wedge})$, which has the law of the $\rm{Airy}_2$ process minus a parabola.
Its distribution at $x = 0$ is the GUE Tracy-Widom law.

The soft shock also interpolates between two $\rm{Airy}_1$ processes at $x = -\infty$ and $x = \infty$.
Indeed, affine and translation symmetries of the KPZ fixed point \cite[Theorem 4.5]{MQR} imply that for constants $c$ and $u$,
$$\h \big(1,x+u; h_0(y) \big) -c(x+u)-\frac{c^2}{4} \;\overset{law}{=}\;
\h \Big(1,x; \, h_0 \big (y+u+\frac{c}{2} \big) - c \big (y+u+\frac{c}{2} \big ) \Big).$$
Thus, $\h(1,x \pm L; \, 2\beta|y|) - \beta^2 \mp 2\beta(x \pm L)$ has the same law as $\h (1,x; \,4 \beta (x \pm (\beta+L))_{\mp})$.
The latter processes converge to $\h(1,x;\, \text{flat})$ as $L \to \infty$.

One can also find the $\rm{Airy}_{2 \to 1}$ process, which is the law of $\h \big (1,x; \, - \infty \cdot \ind{y < 0} \big )$.
The initial data is called half-flat. Indeed, $\h (1,x+\beta; \, 2\beta|y|) - \beta^2 + 2\beta(x+\beta)$ has the law of
$\h (1,x; \, 4\beta (y)_{-})$, and $4\beta (y)_{-}$ converges to the half-flat function as $\beta \to - \infty$.

\subsection{An overview of the proof} \label{sec:overview}
It is well known that the correlation functions of TASEP, which provide the probability of
particles being at specific sites, are determinantal; see for instance \cite{BFPS, BG} and references there.
Our proofs rely on such formulae. 

Let us summarize how the GOE TW-squared law arises in Theorem \ref{thm:1}.
The operator $K_{\beta}$ whose Fredholm determinant provides the law of $h(1,0;2\beta |y|)$ can be \emph{factorized} as
$$I-K_{\beta} = (I - M_{\beta} K_0M^{-1}_{\beta}) \cdot (I - M^{-1}_{\beta} K_0 M_{\beta}) + \rm{Err}_{\beta}.$$
$\rm{Err}_{\beta}$ is an error term that is vanishingly small in the appropriate trace norm
as $\beta \to \infty$. This effectively allows us to consider the Fredholm determinant of the
product. That approximately becomes a product of determinants, and then the conjugations by
$M_{\beta}$ can be removed. This results in the GOE TW-squared law in the large $\beta$ limit.

Observe that if one conjugates away $M_{\beta}$ from one of the factors in the above representation
then the other factor is conjugated by $M^2_{\beta}$, and the resulting operator,
$M^{\pm2}_{\beta}K_0 M^{\mp2}_{\beta}$, does not in fact converge as $\beta \to \infty$.
This was a challenge faced in previous works.

\subsection{Review of literature} \label{sec:lit}
The study of the TASEP shock has a history and the reader may find nice discussions in \cite{F2, FFV}
and their references. We provide an overview of prior works most directly related to ours.

In \cite{BFS} the authors find determinantal formulae for TASEP with particles having
varying speeds, which allows them to study shock fluctuations with Bernoulli-random initial
data. The fluctuations there are Gaussian to the order of $t^{1/2}$. The papers \cite{BC, CFP, FFV} have
related results for Bernoulli initial data. Deterministic shock-like initial data is studied in \cite{FN1, FN3}
by connecting TASEP to last passage percolation. The authors prove that shock fluctuations
for various setups are governed by the maximum of various Tracy-Widom random variables,
although they are unable to treat the step initial density \eqref{eqn:ID}.

The soft shock is introduced in \cite{FN2} in a setup where particles move at two different speeds
instead of being spread out with the two densities $\rho_{\pm}$. The authors prove the analogue of our 
Theorem \ref{thm:2}, and present determinantal kernels for the multi-point distributions
in terms of contour integrals. One may verify that their kernel matches ours. They also
conjecture our Theorem \ref{thm:1}. A beautiful illustration of the convergence of
$X_t(n^{\rm{shk}}_t)$ to the GOE TW-squared law is shown in \cite[Figure 1]{FN2}. The paper
\cite{N} also considers a scenario like the soft shock but with narrow-wedge-like initial data.

Finally, it is well known that for the general asymmetric simple exclusion process, when started
with Bernoulli-random initial data modelling the step density \eqref{eqn:ID}, a second class particle
from the origin follows the macroscopic shock for large times $t$ and has asymptotically Gaussian
fluctuations to the order $t^{1/2}$; see e.g.~\cite{Lig}. However, when \eqref{eqn:ID} is modelled
by deterministic initial data, \cite{FGN} proves that the asymptotic position of the second class particle
in TASEP is the difference of two independent GOE Tracy-Widom random variables on the scale of $t^{1/3}$.
One may think of the second class particle as a random walk in the potential well given by the
TASEP height process, and so it should sit at the minimum of process \eqref{limitprocess},
which is indeed $(X_{\rm{TW}_1}- X'_{\rm{TW}_1})/2^{5/3}$. Further discussions about shocks in ASEP
have appeared in \cite{BB, N2} after this paper.

\subsection*{Acknowledgements}
JQ was supported by the Natural Sciences and Engineering Research Council of Canada.
MR is thankful to Alexei Borodin for helpful discussions and guidance during early
stages of this work. The authors also thank Patrik Ferrari for a valuable discussion as well as
Ivan Corwin, Promit Ghosal and Daniel Remenik for their comments.

\section{The soft shock operator} \label{sec:softkernel}

The operator $K_{\beta}$ associated to the soft shock is defined in terms of operators
\begin{equation} \label{eqn:S}
S_x = \exp \{x \partial^2 + \frac{1}{3}\partial^3\}, \;\;\text{for}\; x \in \R,
\end{equation}
acting on functions $f \in L^2(a,\infty)$ for any fixed  $a >-\infty$.
(Recall the notation $e^L = \sum_{k \geq 0} \frac{L^k}{k!}$. For instance,
translation by $\lambda$ is $e^{\lambda \partial}$ since
$f(x + \lambda) = \sum_{k \geq 0} \frac{\partial^k f(x)}{k!} \lambda^k$;
its integral kernel is $e^{\lambda \partial}(u,v) = \delta_{u+\lambda,v}$.)

The operator $\exp \{x \partial^2\}$, which corresponds to the heat kernel, is ill-defined
for $x < 0$ but $S_x$ is well-defined due to the presence of the third derivative operator.
In terms of integral kernels,
\begin{equation} \label{eqn:Skernel}
S_x(u,v) = e^{\frac{2}{3}x^3 + x(v-u)} \Ai(v-u + x^2),
\end{equation}
where $\Ai(z)$ is the Airy function defined as
\begin{equation} \label{eqn:Aifunc}
\Ai(z) = \frac{1}{2 \pi \mathbf{i}} \, \oint_{\langle} dw \, e^{\frac{w^3}{3} - z w},
\end{equation}
and $\langle$ is a contour consisting of two rays going from $e^{-\mathbf{i} \pi/3} \infty$
to $e^{\mathbf{i} \pi/3} \infty$ through 0. (The Airy function also satisfies the
Airy equation $\Ai '' (x) = x \Ai(x)$ with $\Ai(x) \to 0$ as $x \to + \infty$.)
The operator $S_0$ will often be denoted $S$. We will use the fact that
$$ S^{*} S = S S^{*} = I.$$

We now introduce the important \emph{hitting operator}.
Let $B(y)$, for $y \geq 0$, denote a Brownian motion with diffusion coefficient 2.
Let $h : [0,\infty) \to [-\infty, \infty)$ be upper semicontinuous with at most linear growth
in the sense that $h(y) \leq C (1+ |y|)$ for some constant $C$. Let
$$ \tau = \inf \, \{ y \geq 0: B(y) \leq h(y)\}.$$
Define the operator $S^{\rm{hypo}(h)}_x$ in terms of its integral kernel as
\begin{equation} \label{eqn:hypokernel}
S^{\rm{hypo}(h)}_x(u,v) = \E{S_{x-\tau}(B(\tau),v)\ind{\tau < \infty} \,|\, B(0) = u}.
\end{equation}
If $u \leq h(0)$ then $S^{\rm{hypo}(h)}_x = S_x$. If $h$ is continuous and $u > h(0)$ then
$$ S^{\rm{hypo}(h)}_x(u,v) = \E{S_{x-\tau}(h(\tau),v)\ind{\tau < \infty} \,|\, B(0) = u}.$$
Consider also the projection operators onto $L^2(a,\infty)$ and $L^2(-\infty, a)$, respectively:
\begin{equation} \label{eqn:proj}
\chi_a(u,v) = \ind{u = v, \,u > a} \quad \text{and} \quad \bar{\chi}_a = 1 - \chi_a.
\end{equation}

The hitting operator is defined as follows.
Consider an upper semicontinuous $h : \R \to [-\infty, \infty)$ that has at most linear growth.
The hitting operator associated to $h$ requires choosing a \emph{split point} $x \in \R$.
Then consider the functions
$$ h^{\pm}_x(y) = h(x \pm y) \quad \text{for}\;\; y \geq 0.$$
The hitting operator is
\begin{equation} \label{eqn:hitkernel}
K^{\rm{hypo}(h)} = I - \left (S_x - S^{\rm{hypo}(h^{-}_x)}_x \right)^{*} \chi_{h(x)}
\left (S_{-x} - S^{\rm{hypo}(h^{+}_x)}_{-x}\right).
\end{equation}
It is a crucial property of the hitting operator that it does not depend on the choice
of split point $x$ (see \cite{QR1} for a proof).

The operator $K_{\beta}$ is the hitting operator associated to $h_{\beta}(y) = 2\beta |y|$.
It is natural (and crucial for the large $\beta$ asymptotics) to take the split point at $x = 0$,
which utilizes the fact that $h_{\beta}$ has different slopes on the two sides of the split point.
Denoting $h^{+}_{\beta}(y) = 2 \beta y$ for $y \geq 0$,
\begin{equation} \label{eqn:Khypo}
I - K_{\beta} = \left ( S - S^{\rm{hypo}(h^{+}_{\beta})}\right)^{*} \chi_0 
\left ( S - S^{\rm{hypo}(h^{+}_{\beta})}\right)\,.
\end{equation}
Since $S^{*} S = I$, $K^{\rm{hypo}(h_{\beta})}$ can be expressed as
$$S^{*} \bar{\chi}_0S \,+\, S^{*} \chi_0 S^{\rm{hypo}(h^{+}_{\beta})} \,+\, (S^{\rm{hypo}(h^{+}_{\beta})})^{*} \chi_0 S \,-\,
(S^{\rm{hypo}(h^{+}_{\beta})})^{*} \chi_0 S^{\rm{hypo}(h^{+}_{\beta})}.$$
Each of these terms have a presence of the operator $\exp \{ \pm \, \partial^3/3 \}$ on both sides.
This ensures that the operator $\exp\{ x \partial^2\}$ can be applied legally around $K^{\rm{hypo}(h_{\beta})}$
for every $x \in \R$, and so the operator inside the determinant from the statement of Theorem \ref{thm:2} is well-defined.

For general initial data $h_0$, the multi-point distribution functions of $\h(1,x; h_0)$ are given as follows.
Given $x_1 < \ldots < x_m$ and $a_1, \ldots, a_m$,
\begin{align} \label{eqn:KPZjoint}
& \pr{\h (1,x_i; h_0) \leq a_i\,;\, 1 \leq i \leq m} = \\
\nonumber & \det \left( I - e^{-x_m \partial^2} K^{\rm{hypo}(h_0)}e^{x_m \partial^2}
\big (I - e^{(x_1-x_m)\partial^2}\bar{\chi}_{a_1} e^{(x_2-x_1)\partial^2} \bar{\chi}_{a_2}
\cdots e^{(x_m-x_{m-1})\partial^2} \bar{\chi}_{a_m}\big)\right)_{L^2(\R)}.
\end{align}

The determinantal expression for the multi-point distribution function is the `path integral' version from \cite{MQR}.
There is an alternative `extended kernel' version. The hitting operator is also introduced
in \cite{QR1} in a modified form and precursors appear in \cite{BCR, CQR, PS, QR3}.

\section{First limit transition: proof of Theorem \ref{thm:2}} \label{sec:KPZlimit}

Let us introduce a parameter $\eps > 0$ and write
$$ t = 2 \eps^{-3/2}.$$
Then $m(t,x) = 2x \eps^{-1} + 2\beta x \eps^{-1/2} + O(\beta^2)$.
The events of interest are
$$ X_{2 \eps^{-3/2}}\left(\frac{1}{2}\eps^{-3/2} -x\eps^{-1} - \frac{\beta^2}{2} \eps^{-1/2}\right) \geq
2x \eps^{-1} + (2\beta x -a)\eps^{-1/2} + O(\beta^2).$$
In order to prove Theorem \ref{thm:2} one must derive the limiting joint probabilities of such events
as $\eps \to 0$. Upon replacing $x$ with $x - (\beta^2/2)\eps^{1/2}$ the event becomes
\begin{equation} \label{eqn:softevent}
X_{2 \eps^{-3/2}}\Big(\frac{1}{2}\eps^{-3/2} -x\eps^{-1}\Big)
\geq 2x \eps^{-1} - \big(\beta^2 - 2\beta x + a \big)\eps^{-1/2} + O(\beta^2).
\end{equation}

We may express the event \eqref{eqn:softevent} in terms of the height function of TASEP.
For TASEP with initial data $X_0$, let
$$X_t^{-1}(u) = \min\, \{n \in \Z: X_t(n) \leq u\}.$$
The height function $h_t: \Z \to \Z$ at time $t$ is
$$h_t(z) = -2 \big ( X_t^{-1}(z-1) - X_0^{-1}(-1) \big) - z.$$
The KPZ-rescaled height function is
\begin{equation} \label{eqn:KPZh}
h^{\eps}(T,x) = \eps^{1/2} \left [ h_{2T\eps^{-3/2}}( 2x\eps^{-1}) + T \eps^{-3/2} \right ].
\end{equation}
In terms of the KPZ-rescaled height function one has
$$h^{\eps}(T,x) \leq a \;\Longleftrightarrow\;
X_{2T\eps^{-3/2}}\Big(\frac{T}{2}\eps^{-3/2} - x\eps^{-1}\Big) \geq 2x\eps^{-1} - a \eps^{-1/2} + X_0(1).$$
In the $\eps \to 0$ limit the probability of the event in \eqref{eqn:softevent}
remains unaffected if the term $O(\beta^2)$ is ignored. Thus, one must show that the limiting
multi-point probabilities
$$ \lim_{\eps \to 0}\, \pr{ h^{\eps}(1,x_i) \leq \beta^2 - 2\beta x_i +a_i, \; 1\leq i \leq k}$$
are given by the formula from Theorem \ref{thm:2}.
(The function $h^{\eps}(0,y)$ converges uniformly to $h_0(y)= 2\beta |y|$.)

Here there are several approaches. In \cite{FN2}, a determinantal formula is derived for these multi-point
probabilities for the soft-shock data in a related setup, where particles to the left of the origin have a
different speed than those to the right. Using their formula, it is not difficult to guess a determinantal
formula for our setup and then check it using the the bi-orthogonalization procedure from \cite{BFPS, Sas}.
On the other hand, \cite[Theorem 2.6]{MQR} provides a formula for any initial data with a rightmost particle.
One can cutoff the soft-shock initial data far to the right and take a limit as the cutoff is removed to get a
determinantal formula for the multi-point probabilities which coincides with the guess. Then by direct
asymptotic analysis of the associated determinantal kernels one arrives at Theorem \ref{thm:2}.
This is done with generality in \cite[Theorem 3.13]{MQR} (``Convergence of TASEP").
Since the limiting kernel is the same as \cite{FN2}, we omit the details.

\section{Second limit transition: proofs of Theorem \ref{thm:1} and \ref{thm:3}} \label{sec:mainthm}

The proof of Theorem \ref{thm:1} is presented in Section \ref{sec:thm1} followed
by the proof of Theorem \ref{thm:3} in Section \ref{sec:thm3} since the latter builds on the former.
We first define the GOE Tracy-Widom law, introduced in \cite{TW}, in a suitable form.
For the remainder of the paper it is assumed that $\beta \geq 0$.

\paragraph{\textbf{The GOE Tracy-Widom law}}
The distribution function of the GOE Tracy-Widom law may be written as a Fredholm determinant \cite{FS}.
Consider the operator $A$ with integral kernel $A(u,v) = 2^{-1/3}\Ai\big(2^{-1/3}(u+v)\big)$.
If $R$ is the reflection operator:
\begin{equation} \label{eqn:R}
 Rf(x) = f(-x),
\end{equation}
then $A$ may be expressed as $A = R S^2 = S^{*}RS$.
This representation uses that $S^2 = e^{2\partial^3/3}$ and, as an integral kernel,
$$\exp \Big \{\frac{t}{3} \partial^3 \Big \}\,(u,v) = t^{-1/3} \Ai \big(t^{-1/3}(v-u)\big)\;\;\text{for}\; t > 0.$$
That $RS^2 = S^{*}RS$ is implied by the relation $\partial R = - R \partial$. It will turn out that $A$ is the
operator $K_{0}$. The GOE Tracy-Widom distribution function is
\begin{equation} \label{eqn:AGOE}
F_1\big(2^{2/3}a \big) = \pr{X_{\rm{TW}_1} \leq 2^{2/3} a} = \det (I - \chi_{a} \,A \, \chi_{a})_{L^2(\R)}\, .
\end{equation}

\subsection{Proof of Theorem \ref{thm:1}} \label{sec:thm1}

Let $M_{\beta}$ denote the multiplication operator:
$$M_{\beta}f(x) = e^{\beta x} f(x).$$
Note also the translation operator $f \mapsto f(x+\lambda)$ is given by $e^{\lambda \partial}$.

\begin{lem} \label{lem:commutator}
The following commutation relations hold between $M_{\beta}$, $S$, $R$ and the translation
operator.
\begin{enumerate}
\item $M_{\beta} S = \exp \big \{ \frac{1}{3}(\partial - \beta)^3 \big \} M_{\beta}$,
\item $M_{\beta} \exp \{ \lambda \partial^2 \} = \exp \{ \lambda (\partial - \beta)^2\} M_{\beta}$,
\item $M_{\beta} \exp \{ \lambda \partial \} = \exp \{ \lambda (\partial - \beta)\} M_{\beta}$,
\item $M_{\beta}R = R M_{-\beta}$ \;and\; $\exp \{\lambda \partial\} S = S \exp \{\lambda \partial \}$.
\end{enumerate}
\end{lem}

\begin{proof}
Relations (1) -- (3) follow from the identity $\partial M_{\beta} = M_{\beta} (\partial + \beta)$.
Relation (4) is clear.
\end{proof}

The following lemma is key to calculating the hitting operator associated to $h_{\beta}(y) = 2 \beta |y|$.
\begin{lem}[Reflection lemma] \label{lem:Shypo}
Let $h^{+}_{\beta}(y) = 2 \beta y$ for $y \geq 0$. Then the operator
$$S^{\rm{hypo}(h^{+}_{\beta})} = \chi_0 \big(M_{\beta} R M_{-\beta} \big) S + \bar{\chi}_0 S.$$
\end{lem}

\begin{proof}
Recall that $S^{\rm{hypo}(h^{+}_{\beta})}(u,v) = S(u,v)$ if $u \leq h^{+}_{\beta}(0) = 0$.
This contributes the term $\bar{\chi}_0 S$. Now assume that $u > 0$ and let $\tau$ be the
hitting time of a Brownian motion of diffusion coefficient 2, starting from $u$,
to the hypograph of $h^{+}_{\beta}$.

Observe that $S_{-t}(2\beta t, v) = e^{-2t^3/3 -t(v-2\beta t)}\Ai(v-2\beta t + t^2)$.
Recall that the Airy function has the following decay: there is a constant $C$ such that
\begin{equation} \label{eqn:Airydecay}
 |\Ai(z)| \leq C\;\;\text{if}\; z \leq 0 \;\text{and}\; |\Ai(z)| \leq C e^{-\frac{2}{3} z^{3/2}} \; \text{if}\; z > 0.
\end{equation}
The above implies that $S_{-t}(2\beta t, v)$ decays sufficiently fast that one has
$$ S^{\rm{hypo}(h^{+}_{\beta})}(u,v) = \lim_{T \to \infty}\, \E{ S_{-\tau}(2 \beta \tau, v) \ind{\tau \leq T} \,|\, B(0) = u}.$$

For $t \leq T$, $S_{-t} = e^{(T-t)\partial^2} S_{-T}$, and one recognizes the integral kernel of
$e^{(T-t)\partial^2}$ at the transition density of Brownian motion (with diffusion constant 2) to
go from $B(t) = u$ to $B(T) = v$. So the strong Markov property implies
$$\E{ e^{(T-\tau)\partial^2}(2\beta \tau, v) \,|\, B(0)=u} = \pr{ \tau \leq T,\, B(T) \in dv \,|\, B(0) = u} / dv,$$
where the expression on the right is the transition density of $B$ to go from $B(0) = u$
to $B(T) = v$ while hitting the curve $h^{+}_{\beta}$. Denote this expression $P^{\rm{hit}}_T(u,v)$.
Thus,
$$\E{ S_{-\tau}(2 \beta \tau, v) \ind{\tau \leq T} \,|\, B(0) = u} = P^{\rm{hit}}_T \cdot S_{-T}(u,v).$$ 

Let $X(t) = B(t) - h^{+}_{\beta}(t)$. Then $P^{\rm{hit}}_T$ is the transition density of
$X$ to go from $X(0) = u$ to $X(T) = v - 2\beta T$ while hitting $0$. By the Cameron-Martin
Theorem, $X$ becomes Brownian motion on $[0,T]$, started from $u$ and with diffusion constant 2,
after a change of measure by the density $\exp \{ -\beta (B(T)-u) - \beta^2T\}$.
Consequently,
\begin{align*}
P^{\rm{hit}}_T(u,v) &=
\E{e^{-\beta (B(T)-u) -\beta^2T} \cdot \mathbf{1}\big \{B \;\text{hits 0 on}\; [0,T], \, B(T) \in d(v-2\beta T) \big\} \,\big |\, B(0)=u} /dv \\
& = e^{\beta(u-v)+ \beta^2 T} \, \pr{ B \; \text{hits 0 on}\; [0,T], \, B(T) \in d(v-2\beta T) \,|\, B(0) = u} / dv.
\end{align*}

Since $u > 0$, if $v - 2\beta T \leq 0$ then the latter transition density is
the transition density of $B$ to go from $B(0) = u$ to $B(T) = v - 2\beta T$.
If $v - 2\beta T > 0$, however, one \emph{reflects} along the time axis the initial
segment of $B$ till the time it hits zero. The reflection principle then implies that the latter transition
density is of $B$ to go from $B(0) = -u$ to $B(T) = v-2\beta T$. Hence, for $u > 0$,
\begin{align*}
e^{\beta(v-u)-\beta^2T} P^{\rm{hit}}_T(u,v) &= 
e^{T \partial^2}(u, v - 2\beta T) \bar{\chi}_{2\beta T}(v) + e^{T\partial^2}(-u, v-2\beta T) \chi_{2\beta T}(v) \\
& = e^{T\partial^2 + 2\beta T\partial} \cdot \bar{\chi}_{2\beta T}\,(u,v) + R \cdot e^{T\partial^2 + 2\beta T\partial} \cdot \chi_{2\beta T}\,(u,v).
\end{align*}
Relation (2) of Lemma \ref{lem:commutator} gives $e^{T(\partial + \beta)^2}M_{-\beta} = M_{-\beta}e^{T\partial^2}$.
Consequently, writing $\chi_{2\beta T}$ as $1- \bar{\chi}_{2\beta T}$ and expressing everything in operator notation, we infer
\begin{align*}
\chi_0 P^{\rm{hit}}_T &= \chi_0 M_{\beta} R e^{T(\partial+\beta)^2}M_{-\beta} +
\chi_0 M_{\beta}(I-R)e^{T(\partial+\beta)^2} \chi_{2\beta T} M_{-\beta} \\
&= \chi_0 (M_{\beta}RM_{-\beta})e^{T \partial^2} + \chi_0 M_{\beta}(I-R)M_{-\beta}\,e^{T\partial^2}\bar{\chi}_{2\beta T}.
\end{align*}
On multiplying by $S_{-T}$,
$$\chi_0 P^{\rm{hit}}_T \cdot S_{-T} = \chi_0 (M_{\beta}RM_{-\beta})S + 
\chi_0 M_{\beta}(I-R)M_{-\beta}\,e^{T\partial^2}\bar{\chi}_{2\beta T} \, S_{-T}.$$

The operators $\chi_0$ and $M_{\pm \beta}$ are diagonal, $R$ is anti-diagonal, and none depend on $T$.
The lemma thus follows if for every choice of $u$ and $v$, the quantity
$$e^{T\partial^2} \cdot \bar{\chi}_{2\beta T} \cdot S_{-T}\,(u,v) \; \to\, 0 \;\;\text{as}\;\; T \to \infty.$$
Let $(I)$ denote this quantity. Using the integral kernels of $e^{T\partial^2}$ and $S_{-T}$ one infers that $(I)$ equals
$$\int_{-\infty}^0 dz \frac{1}{\sqrt{4 \pi T}} \exp \big \{-\frac{(z+2\beta T - u)^2}{4T} - \frac{2}{3}T^3
+ T(z+2 \beta T-v) \big\} \cdot \Ai(v-z-2\beta T + T^2).$$

In order to evaluate $(I)$, write the Airy function in terms of its contour integral representation \eqref{eqn:Aifunc}
and switch the contour integration with the integration over variable $z$ by Fubini.
The integral over $z$ is a Gaussian integral, which equals
\begin{equation*}
\int_{-\infty}^0 dz \,  e^{-\frac{(z-u)^2}{4T}+z(w+T-\beta)} =
\sqrt{4 \pi T}  \, e^{ u(w+T-\beta) + T(w+T-\beta)^2} \Phi \Big(-\sqrt{2T}(w+T-\beta) - \frac{u}{\sqrt{2T}} \Big).
\end{equation*}
Here $\Phi(w) = (2 \pi)^{-1/2} \int_{-\infty}^w ds \,e^{- s^2/2}$, where $w$ is a
complex argument and the integral is over the horizontal contour $s \mapsto w + s$, for $s \leq 0$, oriented
from $-\infty$ to $w$. Substituting this into the expression $(I)$, simplifying, and changing variables $w \mapsto w-T$
shows that
$$ I = \frac{1}{2 \pi \mathbf{i}}\, \oint_{\langle+T} dw\, e^{ \frac{1}{3}w^3 -(v-u)w}\,
\Phi \Big (-\sqrt{2T}(w-\beta) - \frac{u}{\sqrt{2T}} \Big ).$$
The contour $\langle +T$ may be shifted back to $\langle$ without changing the integral. Then changing variables
$w \mapsto w + \beta$, and shifting the contour $\langle \,-\, \beta$ back to $\langle$, implies
$$ I = \frac{1}{2 \pi \mathbf{i}}\, \oint_{\langle} dw\, e^{ \frac{1}{3}(w+\beta)^3 -(v-u)(w+\beta)}\,
\Phi \Big (-\sqrt{2T}w - \frac{u}{\sqrt{2T}} \Big ).$$

If the contour is arranged such that $| \mathrm{arg}(w)| = \pi/5$ then $\Phi(-\sqrt{2T}w - u / \sqrt{2T}) \to 0$
as $T \to \infty$. This is because $\Phi(-w) \to 0$ as $w \to \infty$ within the sector
$|\mathrm{arg}(w)| \leq \pi/4 - \eps$ for any $\eps > 0$ \cite[Eq.~7.2.4]{Nist}. Moreover, if $|\mathrm{arg}(w)| \geq \pi/6 +\eps$
along the contour then the exponential factor decays in modulus to the order $\exp \{-\delta \Re(w)^3\}$ for some $\delta > 0$.
Arranging the contour as such, the dominated convergence theorem implies that $(I) \to 0$ as $T \to \infty$.
\end{proof}

We may now observe that the hitting operator $K_0$ is in fact the operator $A$ associated
to the GOE Tracy-Widom law. Employing the definition from \eqref{eqn:Khypo}, Lemma \ref{lem:Shypo},
and using the fact $S^{*}S = I$, it follows that
\begin{align*}
K_0 &= I - (S- S^{\rm{hypo}(h^{+}_0)})^{*} \chi_0 (S- S^{\rm{hypo}(h^{+}_0)}) \\
\nonumber &= S^{*} [ I - (I-R)\chi_0(I-R)] S.
\end{align*}
The relations $R \chi_0 = \bar{\chi}_0 R$ and $R^2 = I$ imply
$$I - (I-R)\chi_0(I-R) = \chi_0 R + R \chi_0 = R,$$
which establishes the claim.

Theorem \ref{thm:2} gives
\begin{align} \label{eqn:h10a}
\pr{\h(1,0; 2\beta |y|) \leq \beta^2 + a} &= \det \left ( I - K_{\beta} \right)_{L^2(\beta^2 + a,\,\infty)} \\
\nonumber &= \det \left ( I - e^{\beta^2 \partial} K_{\beta} e^{-\beta^2 \partial} \right)_{L^2(a,\,\infty)}.
\end{align}

\begin{lem}[Factorization lemma] \label{lem:factor}
The operator $e^{\beta^2 \partial} K_{\beta} e^{-\beta^2 \partial}$ admits the following factored form.
$$I - e^{\beta^2 \partial} K_{\beta} e^{-\beta^2 \partial} =
\left (I - M_{\beta}(A+ E_{\beta})M_{-\beta} \right )^{*} \left(I - M_{\beta}(A+ E_{\beta})M_{-\beta} \right),$$
where $E_{\beta} = S_{-\beta}^{*} \bar{\chi}_0 (I-R) S_{\beta}$ and $A = S^{*} R S$.
\end{lem}

\begin{proof}
Lemma \ref{lem:Shypo} and the relation $\chi_0 = \chi_0 S S^{*} \chi_0$ imply that
\begin{equation} \label{eqn:factor}
I - K_{\beta} = \left [ S^{*}\chi_0(S- \chi_0 M_{2\beta}RS - \bar{\chi}_0 S) \right]^{*}
\left [S^{*}\chi_0(S- \chi_0 M_{2\beta}RS - \bar{\chi}_0 S)\right ].
\end{equation}
Since $S^{*}S = I$ and $M_{\beta}$ commutes with the projection $\chi_0$, we see that
\begin{align*}
S^{*}\chi_0(S- \chi_0 M_{\beta}RM_{-\beta}S - \bar{\chi}_0 S) &= I - S^{*}\bar{\chi}_0 S - S^{*} \chi_0 M_{\beta}RM_{-\beta}S \\
&= I - S^{*}\bar{\chi}_0 S + S^{*} \bar{\chi}_0 M_{\beta} R M_{-\beta} S - S^{*} M_{\beta} R M_{-\beta} S \\
& = I - S^{*} M_{\beta} R M_{-\beta} S - S^{*}M_{\beta}\bar{\chi}_0(I - R)M_{-\beta}S\,.
\end{align*}

We now conjugate the above equation by the translation $e^{\beta^2 \partial}$ and use relations (1) and (3)
from Lemma \ref{lem:commutator} to bring $M_{\pm \beta}$ to the outside. The adjoint of relation (1)
gives $S^{*}M_{\beta} = M_{\beta} \exp \big \{ -\frac{1}{3} (\beta + \partial)^3 \big \}$. Thus, for the term
$S^{*} M_{\beta} R M_{-\beta} S$,
\begin{align*}
e^{\beta^2 \partial} S^{*} M_{\beta} R M_{-\beta} S e^{-\beta^2 \partial} &=
M_{\beta} \exp \left \{ \beta^2(\partial+\beta) - \frac{1}{3} (\beta + \partial)^3 \right \} R \, \times \\
& \quad \times \, \exp \left \{ -\beta^2(\partial+\beta) + \frac{1}{3} (\beta + \partial)^3 \right \} M_{-\beta} \\
& = M_{\beta} \exp \left \{ -\beta \partial^2 - \frac{1}{3} \partial^3 \right \} R
\exp \left \{ \beta \partial^2 + \frac{1}{3} \partial^3 \right \} M_{-\beta} \\
& = M_{\beta} R \exp \left \{ -\beta \partial^2 + \frac{1}{3} \partial^3 + \beta \partial^2 + \frac{1}{3} \partial^3\right \} M_{-\beta} \\
& = M_{\beta} A M_{\beta}.
\end{align*}
The last equation used that $A = RS^2$. A key point above is that conjugation by the translation
cancels the term involving $\partial$ in the expansion of $\pm \frac{1}{3}(\beta + \partial)^{3}$.

Analogously, one computes to see that
\begin{align*}
 e^{\beta^2 \partial} S^{*}M_{\beta}\bar{\chi}_0(I - R)M_{-\beta}S e^{-\beta^2 \partial} & = 
 M_{\beta} \exp \left \{ -\beta \partial^2 - \frac{\partial^3}{3}  \right \} \bar{\chi}_0 (I-R) \, \times \\
  & \qquad \times \, \exp \left \{ \beta \partial^2 + \frac{\partial^3}{3}  \right \} M_{-\beta} \\
 & = M_{\beta} E_{\beta} M_{-\beta}.
\end{align*}
In conclusion,
$$ e^{\beta^2 \partial} S^{*}\chi_0(S- \chi_0 M_{2\beta}RS - \bar{\chi}_0 S) e^{-\beta^2 \partial} =
I - M_{\beta} (A + E_{\beta}) M_{-\beta}.$$
The lemma follows from this relation and the expression \eqref{eqn:factor} for $I- K_{\beta}$.
\end{proof}

Lemma \ref{lem:factor} and \eqref{eqn:h10a} give
$$\pr{\h(1,0) \leq \beta^2 + a} = \det \Big( (I - M_{\beta}(A+ E_{\beta})M_{-\beta} )^{*}
(I - M_{\beta}(A+ E_{\beta})M_{-\beta})\Big)_{L^2(a,\, \infty)}\,.$$
Decompose the product above in the form $(I-X)^{*} \chi_a (I-X) + (I-X)^{*} \bar{\chi}_a (I-X)$.
The determinant of the first term over $L^2(a, \infty)$ factorizes, and upon conjugating out $M_{\beta}$ from each
factor one gets
\begin{equation} \label{eqn:h10b}
\det \Big ((I - M_{\beta}(A+ E_{\beta})M_{-\beta} )^{*} \chi_a (I - M_{\beta}(A+ E_{\beta})M_{-\beta}) \Big)_{L^2(a,\infty)} =
\det \Big ( I - A- E_{\beta}\Big)_{L^2(a,\infty)}^2.
\end{equation}

The proof of Theorem \ref{thm:1} will be completed by showing that the second term in the decomposition,
as well as the term $E_{\beta}$, provide negligible error as $\beta \to \infty$. This is the content of
the following two lemmas. The argument makes use of some standard inequalities between the
Fredholm determinant, trace norm, Hilbert-Schmidt norm and operator norm that may be found
in the book \cite{Simon}.

\begin{lem} \label{lem:error1}
As $\beta \to \infty$, $|| E_{\beta}||_{\rm{tr}} \to 0$ on $L^2(a,\infty)$. Consequently,
$$\det \big ( I - A- E_{\beta}\big)_{L^2(a,\infty)} \longrightarrow \, \det \big ( I - A \big)_{L^2(a,\infty)}.$$
\end{lem}

\begin{proof}
The representation of $E_{\beta}$ over $L^2(a,\infty)$ is $\chi_a E_{\beta} \chi_a$.
Therefore, $\chi_a E_{\beta} \chi_a = E^1_{\beta} - E^2_{\beta}$ with
$$ E^1_{\beta} = \chi_a S_{-\beta}^{*} \bar{\chi}_0 S_{\beta} \chi_a\;\; \text{and}\;\;
E^2_{\beta} =  \chi_a S_{-\beta}^{*} \bar{\chi}_0 R S_{\beta} \chi_a.$$
If suffices to show that both of the operators above have vanishingly small trace norm
on $L^2(\R)$ as $\beta \to \infty$.

Consider the operator $E^1_{\beta}$. Using the inequality $|| T_1 T_2 ||_{\rm{tr}} \leq ||T_1||_{\rm{HS}} \,||T_2||_{\rm{HS}}$
with $T_1 = \chi_a S_{-\beta}^{*} \bar{\chi}_0$ and $T_2 = \bar{\chi}_0 S_{\beta} \chi_a$ gives
$$ ||E^1_{\beta}||_{\rm{tr}} \leq || \chi_a S_{-\beta}^{*} \bar{\chi}_0 ||_{\rm{HS}}\, || \bar{\chi}_0 S_{\beta} \chi_a ||_{\rm{HS}}
= || \bar{\chi}_0 S_{-\beta} \chi_a ||_{\rm{HS}}\, || \bar{\chi}_0 S_{\beta} \chi_a ||_{\rm{HS}}.$$
Since $S_{\pm \beta}(u,v) = \exp \{\pm \frac{2}{3}\beta^3 \pm \beta(v-u)\} \Ai(v-u+\beta^2)$, one has
\begin{align*}
|| \bar{\chi}_0 S_{-\beta} \chi_a||^2_{\rm{HS}} \cdot || \bar{\chi}_0 S_{\beta} \chi_a ||^2_{\rm{HS}} = &
\int_0^{\infty} du \, \int_a^{\infty} dv \, e^{-2 \beta(v+u)} \Ai^2(v+u+\beta^2) \; \times \\
& \int_0^{\infty} du \, \int_a^{\infty} dv \, e^{2 \beta(v+u)} \Ai^2(v+u+\beta^2).
\end{align*}

We change variable $v \mapsto v+a$ in both integrals above. Then, changing variables
$y := u+v$ and $x := u-v$ in both integrals gives
$$|| \bar{\chi}_0 S_{-\beta} \chi_a||^2_{\rm{HS}} \cdot || \bar{\chi}_0 S_{\beta} \chi_a ||^2_{\rm{HS}} =
 \int_0^{\infty} dy \, y e^{-2\beta y} \Ai^2(\beta^2 + y +a) \, \int_0^{\infty} dy \, y e^{2\beta y} \Ai^2(\beta^2+y+a).$$
Rescaling the variable of the first integral as $y \mapsto y/2\beta$, and of the second as $y \mapsto \beta^2 y/2$,
shows that
\begin{align} \label{eqn:error1b}
	|| \bar{\chi}_0 S_{-\beta} \chi_a||^2_{\rm{HS}} \cdot || \bar{\chi}_0 S_{\beta} \chi_a ||^2_{\rm{HS}} = &
	\; \frac{\beta^2}{16} \int_0^{\infty}dy\, y\, e^{-y}\, \Ai^2(\beta^2 + a + (2\beta)^{-1}y) \; \times \\ \nonumber
	& \;\; \int_0^{\infty}dy\, y\, e^{\beta^3 y} \, \Ai^2 \big (\beta^2 (1+\frac{y}{2})+ a \big).
\end{align}

Recall there is a constant $C$ such that
$|\Ai(z)| \leq C \exp \{ -\frac{2}{3}z^{3/2} \}$ if $z \geq 0$ and $|\Ai(z)| \leq C$ if $z < 0$.
Since $a$ is fixed, suppose $\beta$ satisfies $\beta^2 + a \geq 1$, say.
Then due to the aforementioned bound on the Airy function the contribution
to the first of the two integrals above results from $y$ being of bounded order,
$y \approx 1$. In particular, there is a constant $C_a$ such that for sufficiently large
$\beta$ (in terms of $a$),
$$\int_0^{\infty}dy\, y\, e^{-y}\, \Ai^2(\beta^2 + a + (2\beta)^{-1}y) \leq C_a \, \Ai^2(\beta^2). $$

The magnitude of the second integral from \eqref{eqn:error1b} may also be determined from
a critical point analysis by using the bound on the Airy function above. By abusing notation a bit,
there is a constant $C_a$ such that for large enough $\beta$,
$$ e^{\beta^3 y} \, \Ai^2 \Big(\beta^2(1 + \frac{y}{2}) + a \Big) \leq C_a\,
\exp \left \{ \beta^3 \, \big (y  - \frac{4}{3} (1+ \frac{y}{2})^{3/2} \big )  \right \}.$$
The function $y - (4/3)(1+ (y/2))^{3/2}$ is uniquely maximized at $y = 0$
and its value there is $-\frac{4}{3}$. Therefore the second integral from $\eqref{eqn:error1b}$
is of order $e^{-\frac{4}{3}\beta^3}$ as $\beta \to \infty$. Consequently, there is a (new)
constant $C_a$ such that such for sufficiently large $\beta$,
$$|| \bar{\chi}_0 S_{-\beta} \chi_a||^2_{\rm{HS}} \cdot || \bar{\chi}_0 S_{\beta} \chi_a ||^2_{\rm{HS}} \leq
C_a \, \beta^2 \, \Ai^2(\beta^2) e^{-\frac{4}{3}\beta^3}.$$
This shows that $||E^1_{\beta}||_{\rm{tr}} \to 0$ as $\beta \to \infty$
since $\Ai(\beta^2)$ is of order $e^{-\frac{2}{3}\beta^3}$.

Now consider the operator $E^2_{\beta}$. Using the definitions one has that
$$E^2_{\beta}(u+a,v+a) = \ind{u\geq 0, \, v\geq 0} \, e^{\beta(v-u)}
\int_0^{\infty} dz\, e^{-2\beta z} \Ai(\beta^2 +a +u +z) \Ai(\beta^2+a +v-z).$$
The trace norm of $E^2_{\beta}$ is the same as that of $(u,v) \mapsto E^2_{\beta}(u+a,v+a)$ since
the latter is a conjugation of the former by the unitary operation of translation.
So we consider the latter kernel.

When $\beta^2 + a \geq 1$, the major contribution to the integral above comes from $z$ being in a
region around zero, $z \approx 0$, due to the rapid decay of the integrand in the variable $z$.
Consequently, for large $\beta$ there is a constant $C_a$ such that
\begin{equation} \label{eqn:error1c}
|E^2_{\beta}(u+a, v+a)| \leq C_a\, \ind{u \geq 0, \, v \geq 0} \, e^{\beta(v-u)} \Ai(u+\beta^2) \Ai(v+\beta^2).
\end{equation}
The right hand side above decays rapidly in the variable $u$, namely, it is at most of order
$e^{-\frac{2}{3} (\beta^3 + u^{3/2})}\, \ind{u \geq 0}$. Consider its rate of decay in the variable $v$.

The asymptotics of the Airy function show that for $v \geq 0$,
$$ |e^{\beta v} \Ai(v+\beta^2)| \leq \exp \left \{ \beta v - \frac{2}{3}(v+\beta^2)^{3/2} + \mathrm{const} \right \}.$$
The exponent above is uniquely maximized at $v = 0$ whereby it equals $-\frac{2}{3}\beta^3$. Moreover,
for large values of $v$ the exponent is of order $\beta v -\frac{2}{3} v^{3/2} -\beta^2 v^{1/2} + \mathrm{const}$,
which is seen from a Taylor expansion of $(1+x)^{3/2}$ around $x =0$. If $\beta \geq 1$, say,
then the quantity $\beta v -\frac{2}{3} v^{3/2} -\beta^2 v^{1/2} \leq - \frac{5}{8} v^{1/2}$
because $\beta^2 + \frac{2}{3} v - \beta v^{1/2}$ is at least $\frac{5}{8}\beta^2$, the minimum being
at $v = \frac{9}{16}\beta^2$. The upshot is that for $v \geq 0$,
$$|e^{\beta v} \Ai(v+\beta^2)| \leq \exp \left \{ - \frac{2}{3}\beta^3 -\frac{5}{8}v^{1/2} + \mathrm{const} \right \}.$$

All in all it follows that there are constants $C_a$ and $\kappa > 0$ such that
$$ |E^2_{\beta}(u+a,v+a)| \leq C_a\, \ind{u \geq 0,\, v \geq 0}\, e^{-\frac{4}{3}\beta^3 - \kappa\, (u^{3/2} + v^{1/2})}.$$
This implies that the trace norm of $E^2_{\beta}$ decays to the order $e^{-\frac{4}{3}\beta^3}$, as required.
\end{proof}

\begin{lem} \label{lem:error2}
As $\beta \to \infty$, the difference of determinants
$$\det \Big( (I - M_{\beta}(A+ E_{\beta})M_{-\beta} )^{*}
(I - M_{\beta}(A+ E_{\beta})M_{-\beta})\Big)_{L^2(a,\, \infty)} -
\det \big ( I - A- E_{\beta}\big)_{L^2(a,\infty)}^2$$
tends to zero.
\end{lem}

\begin{proof}
In the following argument all Fredholm determinants are over $L^2(a,\infty)$.
Denote $X = A + E_{\beta}$. On the space $L^2(a,\infty)$,
$$(I- M_{\beta}XM_{-\beta})^{*} \bar{\chi}_a (I- M_{\beta}XM_{-\beta}) =
(M_{\beta} X M_{-\beta})^{*} \bar{\chi}_a (M_{\beta}XM_{-\beta})$$
because $\bar{\chi}_a$ annihilates the identity on $L^2(a,\infty)$.
Consequently, on $L^2(a,\infty)$,
\begin{align*} (I - M_{\beta}XM_{-\beta})^{*}(I - M_{\beta}XM_{-\beta}) = & 
\underbrace{\chi_a(I - M_{\beta}XM_{-\beta})^{*}\chi_a(I - M_{\beta}XM_{-\beta})\chi_a}_{Y} \; + \\
& + \underbrace{\chi_a(M_{\beta} X M_{-\beta})^{*} \bar{\chi}_a (M_{\beta}XM_{-\beta})\chi_a}_{E}.
\end{align*}
The determinant of $Y$ is $\det(I-X)^2$.

Since $\chi_a$ and $\bar{\chi}_{a}$ are projections and commute with $M_{\pm \beta}$,
\begin{align*}
Y  &= M_{-\beta} (\chi_a - \chi_a X^{*} \chi_a) \chi_a M_{2\beta} (\chi_a - \chi_a X \chi_a) M_{-\beta}\, ,\\
E  &= M_{-\beta} (M_{\beta}\bar{\chi}_aX\chi_a)^{*} (M_{\beta}\bar{\chi}_aX\chi_a) M_{-\beta}\,.
\end{align*}
The operators $I - X$ and $I - X^{*}$ are invertible on $L^2(a,\infty)$
for sufficiently large $\beta$ because $I-A$ is invertible there (since $\det(I-A)_{L^2(a,\infty)} = F_1(2^{2/3}a) > 0$)
and $E_{\beta}$ has vanishingly small trace norm as $\beta \to \infty$.
In fact, this means that both the operator norm and the Fredholm determinant of $I -  X$
are uniformly bounded away from 0 for sufficiently large $\beta$. This implies invertibility of
$Y$ on $L^2(a,\infty)$, and one observes from the above expressions for $Y$ and $E$ that on this space
\begin{equation} \label{eqn:error2a}
M_{-\beta} Y^{-1}E M_{\beta} = (I-\chi_aX\chi_a)^{-1} \chi_a M_{-2\beta}(I-\chi_a X^{*} \chi_a)^{-1}
(M_{\beta}\bar{\chi}_aX\chi_a)^{*} (M_{\beta}\bar{\chi}_aX\chi_a).
\end{equation}

In order to compare the determinant of $Y+E$ with that of $Y$ one first conjugates
both operators as $M_{-\beta}(Y+E) M_{\beta}$ and $M_{-\beta} E M_{\beta}$,
and then uses the inequality 
$$|\det(I-T) - 1| \leq ||T||_{\rm{tr}}\, e^{||T||_{\rm{tr}}+1}, $$
to deduce that
$$ \left | \det( Y + E) - \det(Y) \right| \leq
| \det(Y)| \, || M_{-\beta}Y^{-1} E M_{\beta}||_{\rm{tr}} \, e^{|| M_{-\beta}Y^{-1} E M_{\beta}||_{\rm{tr}}+1} .$$
The determinant of $Y$ remains bounded in $\beta$ by Lemma \ref{lem:error1}.

The trace norm of $M_{-\beta}Y^{-1} E M_{\beta}$ may be bounded using the inequalities
$|| T_1 T_2 ||_{\rm{tr}} \leq || T_1 ||_{\rm{op}}\, || T_2||_{\rm{tr}}$ and
$|| T_1 T_2 ||_{\rm{tr}} \leq || T_1 ||_{\rm{tr}}\, || T_2||_{\rm{op}}$. (The second follows from the first
by taking adjoints.) Thus,
\begin{align} \label{eqn:error2b}
|| M_{-\beta}Y^{-1} E M_{\beta}||_{\rm{tr}} &  \leq ||(I - \chi_a X \chi_a)^{-1} ||_{\rm{op}} \, \times \\ \nonumber
& \qquad \times \, ||\chi_a M_{-2\beta}\, (I-\chi_a X^{*} \chi_a)^{-1} \,
(M_{\beta}\bar{\chi}_aX\chi_a)^{*} (M_{\beta}\bar{\chi}_aX\chi_a)||_{\rm{tr}}  \\ \nonumber
& \leq ||(I - \chi_a X \chi_a)^{-1} ||_{\rm{op}} \; ||\chi_a M_{-2\beta}||_{\rm{tr}} \, \times \\ \nonumber
& \qquad \times\, ||(I-\chi_a X^{*} \chi_a)^{-1} (M_{\beta}\bar{\chi}_aX\chi_a)^{*} (M_{\beta}\bar{\chi}_aX\chi_a)||_{\rm{op}} \\ \nonumber
& \leq ||(I - \chi_a X \chi_a)^{-1} ||_{\rm{op}}^2 \, || \chi_a M_{-2\beta} ||_{\rm{tr}}\,
|| M_{\beta}\bar{\chi}_aX\chi_a||^2_{\rm{op}}.
\end{align}

The first operator norm in the last expression above remains bounded for large $\beta$ as remarked.
Since $\chi_a M_{-2\beta}$ is a diagonal operator, it has trace norm
$$||\chi_a M_{-2\beta} ||_{\rm{tr}} = \int_a^{\infty} du\, e^{-2\beta u} = \frac{e^{-2\beta a}}{2\beta}.$$
The operator norm of $M_{\beta}\bar{\chi}_aX\chi_a$ is at most
$|| M_{\beta} \bar{\chi}_aA\chi_a||_{\rm{op}} + ||M_{\beta}\bar{\chi}_a E_{\beta} \chi_a||_{\rm{op}}$. Observe that
$$M_{\beta} \bar{\chi}_aA\chi_a(u+a,v+a) = e^{\beta a} \cdot
\Big( e^{\beta u} \, 2^{-1/3} \Ai \big(2^{-1/3}(u+v+2a)\big) \ind{u <0, \, v \geq 0} \Big).$$
The operator norm of the kernel inside the big parentheses is bounded in terms of $\beta$ because
the kernel decays to the order $e^{-\frac{\sqrt{2}}{3}v^{3/2}}$ for large values of $v$ and to the
order $e^{-\beta |u|}$ for negative values of $u$. Since the operator displayed above is a conjugation of
$M_{\beta} \bar{\chi}_aA\chi_a$ by a translation, it follows that
$|| M_{\beta} \bar{\chi}_aA\chi_a||_{\rm{op}} \leq C_a\, e^{\beta a}$ for some constant $C_a$.
Similarly, $M_{\beta}\bar{\chi}_a E_{\beta} \chi_a =
e^{\beta a} e^{\partial a}\Big ( M_{\beta} \bar{\chi}_0 e^{-\partial a} E_{\beta} e^{\partial a} \chi_0 \Big) e^{-\partial a}$.
The opertor norm of what sits within the big parentheses is vanishingly small in terms
of $\beta$ by a calculation entirely analogous to that of Lemma \ref{lem:error1}. So in all,
$ || M_{\beta}\bar{\chi}_aX\chi_a||_{\rm{op}} \leq C_a\, e^{\beta a}$ for some constant $C_a$.

Therefore, \eqref{eqn:error2b} implies that for some (new) constant $C_a$,
$$|| M_{-\beta}Y^{-1} E M_{\beta}||_{\rm{tr}} \leq \frac{C_a}{\beta} .$$
Thus $|| M_{-\beta}Y^{-1} E M_{\beta}||_{\rm{tr}}$ tends to $0$ as required.
\end{proof}

Lemma \ref{lem:error1} and Lemma \ref{lem:error2} together conclude the proof of Theorem \ref{thm:1}.

This section concludes by extending Theorem \ref{thm:1} to arbitrary one-point distributions of
$\h(1,x;2\beta|y|)$, which will be utilized in the proof of Theorem \ref{thm:3}.

\begin{prop} \label{prop:1}
For every $a, x \in \R$, as $\beta \to \infty$ the probability
$$ \pr{\h(1, (2\beta)^{-1}x; 2\beta |y|) - \beta^2 \leq a} \longrightarrow F_1\big (2^{2/3}(a+x)\big) F_1\big(2^{2/3}(a-x)\big).$$
\end{prop}

\begin{proof}
By Theorem \ref{thm:2}, the probability
\begin{align*}
\pr{\h(1,x, 2\beta|y|) - \beta^2 \leq a} &=
\det \left ( I - e^{- x \partial^2} K_{\beta} e^{x \partial^2} \right)_{L^2(a+\beta^2,\,\infty)}\\
& = \det \left ( I - e^{- x \partial^2+\beta^2 \partial} K_{\beta} e^{x \partial^2 -\beta^2 \partial} \right)_{L^2(a,\, \infty)}.
\end{align*}
Factorization Lemma \ref{lem:factor} then gives
$$I - e^{- x \partial^2+\beta^2 \partial} K_{\beta} e^{x \partial^2 -\beta^2 \partial}  =
\left (I - e^{x \partial^2} M_{\beta} X M_{-\beta} e^{-x\partial^2} \right )^*
\left (I - e^{-x \partial^2} M_{\beta} X M_{-\beta} e^{x\partial^2} \right),$$
where $X = A + E_{\beta}$. Commutation relation (2) of Lemma \ref{lem:commutator} implies that
\begin{align*}
e^{\mp x \partial^2} \, M_{\beta} X M_{- \beta} \, e^{\pm x\partial^2}
&= M_{\beta} \, e^{\mp x (\partial+\beta)^2} X \, e^{\pm x(\partial+\beta)^2} \, M_{-\beta} \\
& = M_{\beta} \, e^{\mp x( \partial^2 + 2\beta \partial)} X \, e^{\pm x( \partial^2+ 2 \beta \partial)} \, M_{-\beta}.
\end{align*}
The relation $\partial R = - R \partial$ now implies the following identities.
\begin{align*}
e^{\mp x( \partial^2+2 \beta \partial)} \, A \, e^{\pm x(\partial^2 +2\beta \partial)} &= e^{\mp 2\beta x \partial} \, A \, e^{\pm 2\beta x \partial} ,\\
e^{\mp x( \partial^2+2 \beta \partial)} \, E_{\beta} \, e^{\pm x(\partial^2 +2\beta \partial)} & = 
S^{*}_{- \beta \, \mp x} \, e^{\mp 2\beta x \partial} \, \bar{\chi}_0 (I-R) \, e^{\pm 2\beta x \partial} \, S_{\beta \, \pm x} \\
&= S^{*}_{- \beta \, \mp x} \, \bar{\chi}_{\mp 2\beta x}(I - Re^{\pm 4\beta x \partial}) \, S_{\beta \, \pm x}.
\end{align*}

Substituting in $(2\beta)^{-1}x $ for $x$ now shows that $\pr{\h(1,(2 \beta)^{-1}x; 2\beta|y|) - \beta^2 \leq a}$ equals
$$ \det \Big ( \big(I - M_{\beta} (e^{x \partial}Ae^{-x\partial} + E_{\beta,x}) M_{-\beta} \big)^{*}
(I - M_{\beta} \big (e^{-x \partial}Ae^{x\partial} + E_{\beta,-x}) M_{-\beta} \big ) \Big )_{L^2(a,\infty)},$$
where $E_{\beta,x} = S^{*}_{-\beta - x/2\beta} \, \bar{\chi}_{x}(I-Re^{2x \partial}) \, S_{\beta + x/2\beta}$.
The proof now proceeds exactly as in the arguments of Lemma \ref{lem:error1} and Lemma \ref{lem:error2}.
The argument of Lemma \ref{lem:error2} remains the same, and in place of Lemma \ref{lem:error1} one needs
to show that the trace norm of $E_{\beta,\pm x}$ over $L^2(a,\infty)$ converges to zero as $\beta \to \infty$.
The proof of the latter is entirely analogous to the proof of Lemma \ref{lem:error1}. We do not repeat
the calculations for brevity.

In conclusion, as $\beta \to \infty$, $\pr{\h(1,(2\beta)^{-1}x; 2\beta |y|) - \beta^2 \leq a}$ tends to
\begin{align*}
& \det(I - e^{x \partial} A e^{-x\partial})_{L^2(a,\infty)} \cdot \det(I - e^{-x \partial} A e^{x\partial})_{L^2(a,\infty)} = \\
& \det(I - A )_{L^2(a+x,\infty)} \cdot \det(I - A)_{L^2(a-x,\infty)} = F_1(2^{2/3}(a+x))\cdot F_1(2^{2/3}(a-x)).
\end{align*}
\end{proof}

\subsection{Proof of Theorem \ref{thm:3}} \label{sec:thm3}

We will use an argument by way of the variational principle for the law of the process $\h(1,x; 2\beta |y|)$.
An Airy sheet $\A_2(x,y)$ is a random function of real variables $x$ and $y$ defined by the identity
$$ \A_2(x,y) = \h \big (1,x; - \infty \, \ind{z \neq y} \big) + (x-y)^2.$$
Here, $- \infty \,\ind{z \neq y}$ is the narrow wedge at $y$. The height functions above are all coupled
by a ``common noise". This noise is naturally present in TASEP and the coupled 
height functions may be obtained as a joint KPZ scaling limit of TASEPs with
different wedge initial data that all move under a common dynamic. See Section 4.5 of \cite{MQR}.

Actually, \cite{MQR} proves existence of an Airy sheet (due to tightness) but not its uniqueness
(see also \cite{P} for a similar result). Nevertheless, the following properties we use are common to every
Airy sheet: it is continuous, invariant under switching variables, and has the law of the
$\rm{Airy}_2$ process in each variable when the other is held fixed. Also, the following variational principle
applies to every Airy sheet \cite[Theorem 4.18]{MQR} (see also \cite{CQR2}).
\smallskip

\paragraph{\textbf{Variational principle}}
Let $h_0 : \R \to [-\infty, \infty)$ be a upper semicontinuous function with at most linear growth.
Then the KPZ height function $x \mapsto \h(1,x; h_0)$ (as defined in \eqref{eqn:KPZjoint}) satisfies
$$ \h(1,x; h_0) \overset{law}{=} \; \sup_{y \in \R}\, \left \{ \A_2(x,y) - (x-y)^2 + h_0(y) \right \}\,.$$
Variational formulae like these originate in \cite{J} and are similar to the Hopf-Lax-Oleinik formula
for solutions to Burgers' equation; see \cite{CQR2, S3}.

\begin{lem} \label{lem:continuity}
An Airy sheet has the following modulus of continuity uniformly over $y$ and $x_1, x_2$ with $|x_1-x_2| \leq 1$.
$$ \big | \A_2(x_1,y) - \A_2(x_2,y) \big | \leq O_p\big(|x_1-x_2|^{1/4}\big).$$
\end{lem}
The notation $O_p()$ means a random quantity that is finite with probability one.

\begin{proof}
For every fixed $y$, $\A_2(x,y)$ is an $\rm{Airy}_2$ process in $x$, which satisfies the
modulus of continuity estimate stated above by \cite[Theorem 4.4]{MQR}.
(The $\rm{Airy}_2$ process is H\"{o}lder-$(1/2 - \eps)$ almost surely.)
Thus, the modulus of continuity estimate above holds for every fixed $y$.
By a union bound it then holds uniformly over all rational values of $y$.
By continuity of an Airy sheet, it also holds uniformly over all $y$.
\end{proof}

Using the variational principle and separating the supremum over $y \leq 0$
from the supremum over $y \geq 0$ one has that
\begin{align*}
& \h(1,x; 2\beta|y|) \,\overset{law}{=} \,\sup_{y \in \R}\,  \A_2(x,y) - (x-y)^2 + 2\beta|y| = \max\, \{ I, II \}, \;\text{where}\\
& I = \sup_{y \leq 0}\, \A_2(x,y) - (x-y)^2 - 2\beta y \;\;\text{and}\;\; II = \sup_{y \geq 0}\, \A_2(x,y) - (x-y)^2 + 2\beta y \, .
\end{align*}
Rewrite (I) by changing variable $y \mapsto y - \beta + x$ and (II) by changing variable $y \mapsto y + \beta +x$.
Then $\h(1, (2 \beta)^{-1} x; 2\beta |y|) - \beta^2 $ has the law of $\max \left \{ X_1(x) - x, \, X_2(x) + x \right \}$ where
\begin{equation*}
X_1(x) = \sup_{y \leq \beta -\frac{x}{2\beta}}\,  \A_2 \Big(\frac{x}{2\beta}, \, y-\beta+\frac{x}{2\beta}\Big) -y^2 \, ,
\end{equation*}
and
\begin{equation*}
\nonumber X_2(x) = \sup_{y \geq -\beta -\frac{x}{2\beta}}\, \A_2 \Big(\frac{x}{2\beta}, \, y+\beta+\frac{x}{2\beta}\Big) -y^2 \,.
\end{equation*}

Now consider $X_1(x)$ for a fixed value of $x$. Since $y \mapsto \A_2(x/2\beta, y)$ has the law of the $\rm{Airy}_2$ process,
by the modulus of continuity estimate of Lemma \ref{lem:continuity} (the roles of $x$ and $y$ are now switched) one infers that
$$ \sup_{y \, \in \, \left [\beta - \frac{|x|}{2\beta}, \, \beta + \frac{|x|}{2\beta} \right]}\, \A_2\Big(\frac{x}{2\beta}, y-\beta+\frac{x}{2\beta}\Big)\,
= \, \A_2\Big(\frac{x}{2\beta}, \frac{x}{2\beta}\Big) + O_p(\beta^{-1/4}).$$
As a result, the supremum of $\A_2 \Big(\frac{x}{2\beta}, \, y-\beta+\frac{x}{2\beta}\Big) -y^2$
over $y \leq \beta -\frac{x}{2\beta}$ may be replaced by its supremum over $y \leq \beta$
with an additive error of order $o_p(1)$ as $\beta \to \infty$, since the supremum on
the leftover interval is of order $O_p(1) - \beta^2$. (The notation $o_p(1)$ denotes
a term that converges to zero in probability as $\beta \to \infty$.)

Furthermore, due to the modulus of continuity estimate in Lemma \ref{lem:continuity}, the latter
supremum may be replaced by the supremum of the process $y \mapsto \A_2 \Big(0, \, y-\beta \Big)$
over $y \leq \beta$ with an additional penalty of $o_p(1)$. This is because replacing the
$x/2\beta$ by $0$ introduces an additive error of order $O_p(\beta^{-1/4})$.
As a result,
$$X_1(x) = \sup_{y \leq \beta}\, \A_2(0, \, y-\beta) -y^2 + o_p(1)\;\; \text{as}\; \beta \to \infty.$$
This same argument implies that
$$X_2(x) = \sup_{y \geq -\beta}\, \A_2(0,y+\beta) -y^2 + o_p(1)\;\; \text{as}\; \beta \to \infty.$$
Observe that the two suprema above are $X_1(0)$ and $X_2(0)$, respectively.

Since this holds for every fixed $x$, it follows from the variational principle
that for any finite number of points $x_1, \ldots, x_m$, the joint law of the $m$-dimensional vector
$x_i \mapsto \h(1, (2\beta)^{-1} x_i, 2\beta|y|)$ satisfies
\begin{equation} \label{eqn:thm3c}
\h(1, (2\beta)^{-1} x_i; 2\beta|y|) - \beta^2 \,\overset{law}{=}\, \max\, \left \{ X_1(0) -x_i,\, X_2(0) +x_i \right \} + o_p(1)
\end{equation}
as $\beta \to \infty$.

\begin{lem} \label{lem:decorrelation}
The random variables $X_1(0)$ and $X_2(0)$ jointly converge in law to two independent GOE Tracy-Widom
random variables $2^{-2/3} X_{\rm{TW}_1}$ and $2^{-2/3} X'_{\rm{TW}_1}$, respectively, as $\beta \to \infty$.
\end{lem}

\begin{proof}
It suffices to show that given $s, s' \in \R$, as $\beta \to \infty$,
$$ \pr{X_1(0) \leq s, X_2(0) \leq s'} \; \longrightarrow \;F_1(2^{2/3}s) \, F_1(2^{2/3}s').$$
There are numbers $a$ and $x$ such that $s = a+x$ and $s' = a-x$.
Observe the event $\{X_1(0) \leq s, X_2(0) \leq s'\}$ equals the event that $\max \, \{ X_1(0) - x, X_2(0)+x\} \leq a$.
Since $X_1(x) = X_1(0) + o_p(1)$ and $X_2(x) = X_2(0) + o_p(1)$ as $\beta \to \infty$,
it suffices to show that as $\beta \to \infty$,
$$ \pr{\max \, \{ X_1(x) - x, X_1(x) +x\} \leq a} \longrightarrow F_1 \big(2^{2/3}(a+x)\big) F_1\big(2^{2/3}(a-x)\big).$$
The law of the maximum above is $h(1,(2\beta)^{-1}x; 2\beta |y|) - \beta^2$ by the variational principle.
The required convergence is then the statement of Proposition \ref{prop:1}.
\end{proof}

Lemma \ref{lem:decorrelation} together with representation \eqref{eqn:thm3c} imply that as $\beta \to \infty$,
$$\h(1,(2\beta)^{-1}x; 2\beta|y|) - \beta^2 \; \longrightarrow \; \max \, \{ 2^{-2/3}X_{\rm{TW}_1} - x, \, 2^{-2/3} X'_{\rm{TW}_1}+x\}$$
in the sense of finite dimensional laws with respect to the variable $x$.
This completes the proof of Theorem \ref{thm:3}.


\end{document}